\documentclass[12pt,twoside]{amsart}
\usepackage{amssymb}
\usepackage{a4wide}
\usepackage[latin1]{inputenc}
\usepackage[T1]{fontenc}
\usepackage{times}

\def\hpq0{h^{p,q}_{\leq 0}}
\def\Hpq0{\H_{\leq 0}^{p,q}}

\def\C{{\mathbb C}}

\def\H{{\mathcal H}}
\def\E{{\mathcal E}}

\def\Re{{\rm Re\,  }}

\def\be{\begin{equation}}
\def\ee{\end{equation}}

\newtheorem{thm}{Theorem}[section]
\newtheorem{lma}[thm]{Lemma}

\newtheorem{prop}[thm]{Proposition}

\theoremstyle{definition}

\theoremstyle{remark}

\newtheorem{preremark}{Remark}
\newtheorem{preex}{Example}

\numberwithin{equation}{section}

\title[]
{The openness conjecture for plurisubharmonic functions}

\address{Department of Mathematics\\Chalmers University
  of Technology \\
 } 

\email{ bob@chalmers.se}

\author[]{ Bo Berndtsson}

\begin{document}

\begin{abstract}We give a proof of the openness conjecture of Demailly and Koll\'ar. 
\end{abstract}
\maketitle

\section{Introduction}
Let $ u$ be a plurisubharmonic function defined in a neighbourhood of the origin of $\C^n$ such that $e^{-u}$ lies in $L^1$. The {\it openness conjecture}, first proposed by Demailly and Koll\'ar in \cite{Demailly-Kollar}, says that then there is a number $p>1$ such that $e^{-u}$ lies in $L^p$, possibly after shrinking the neighbourhood. This conjecture has attracted a good deal of attention; in particular it has been completely proved in dimension 2 by Favre and Jonsson, \cite{Favre-Jonsson}. 
In arbitrary dimension it  has been reduced to a purely algebraic statement in \cite{Jonsson-Mustata}.

In this paper we will  prove  the openness conjecture in any dimension. 
\begin{thm} Let $u$ be a plurisubharmonic function in the unit ball with $u\leq 0$. Assume that
$$
\int_B e^{-u} <\infty.
$$
Then there is a number $p>1$ such that
$$
\int_{B/2}  e^{-pu} <\infty.
$$
Moreover, $p$ can be taken so that
$$
p\geq 1 +\delta_n/\int_B e^{-u},
$$
where $\delta_n$ depends only on the dimension. 
\end{thm}

The proof of the theorem is inspired by a result from \cite{Berman-Berndtsson}. There we proved that the Schwarz symmetrization, $u^*$,  of an $S^1$-invariant plurisubharmonic function in the ball,$u$,   is again plurisubharmonic. Here $S^1$-invariance means that $u(e^{i\theta}z)=u(z)$ for any $e^{i\theta}$ on $S^1$. Since the Schwarz symmetrization of $u$ is equidistributed with $u$, it holds that
$$
\int_B F(u^*)=\int_B F(u)
$$
for any (measurable) function $F$. Choosing $F(t)=e^{-t}$ and $F(t)=e^{-\epsilon t}$,  this  reduces the openness problem for $S^1$-invariant functions to the case of radial functions. As a consequence one easily sees  that the openness conjecture holds for any $S^1$-invariant plurisubharmonic function in the ball. 

The proof of this, via the symmetrization result from \cite{Berman-Berndtsson}, depends ultimately on a complex variant of the Brunn-Minkowski inequality from \cite{Berndtsson}. The argument can be rephrased in the following form. A consequence of the 'Brunn-Minkowski'-inequality is that
if $u$ is plurisubharmonic and $S^1$-invariant, the volume of the sublevel sets
$$
\Omega(s):=\{z; u(z)<-s\}
$$
is a logconcave function of $s$. The integral of $e^{-u}$ can be written as
\be
\int_0^\infty e^s |\Omega(s)|ds + \omega_n,
\ee
with $\omega_n$ the volume of the unit ball, 
and the logconcavity is the key to studying the convergence of this integral.
In fact, if $|\Omega(s)|$ is logconcave, the integral (1.1) converges (if and) only if $|\Omega(s)|$ decreases like  $e^{-(1+\epsilon)s}$ at infinity (cf Proposition 3.1).

In the situation of Theorem 1.1, instead of looking at just volumes of sets or integrals of functions, we look at the $L^2$-norms on the space of holomorphic functions in the unit ball, induced by our plurisubharmonic weight function $u$ in section 2. We then find a representation of such an $L^2$-norm as an integral over $(0,\infty)$ of weaker norms depending on the variable $s$. These weaker norms have a property analogous to logconcavity - they define an hermitean metric on a certain vector bundle of positive curvature. This positivity property is finally shown in the last section to imply Theorem 1.1.

I would like to thank Mattias Jonsson and Mihai Paun for comments and discussions.
\section{Hermitean norms on $H^2(B,dm)(B)$}

Let $B$ be the unit ball of $\C^n$ and denote by $H_0:=H^2(B,dm)$ the space of  holomorphic functions in $B$, square integrable with respect to Lebesgue measure, $dm$. For $u$, plurisubharmonic and negative in $B$, we also let $H^2(B,e^{-u}dm)$ be the space of holomorphic functions in $B$ such that
$$
\int_B |h|^2 e^{-u}<\infty
$$
(we suppress the Lebesgue measure $dm$ in all integrals over the ball).

For any $s\geq 0$ we put
$$
u_s:=\max (u+s,0)=\max(u, -s) +s.
$$
Note that since $u\leq 0$, $u_0=0$ and  $0\leq u_s\leq s$. Moreover $u_s\geq u+s$. 
For any $s\geq 0$ we define a new norm on $H^2(B,dm)$ by
$$
\|h\|^2_s=\int_B |h|^2 e^{-2u_s}.
$$
Notice the factor 2 in the exponent! By the inequalities for $u_s$ above, we have if $p<2$
\be
\int_B |h|^2 e^{-2u_s}\leq \int_B |h|^2 e^{-pu_s}\leq e^{-ps}\int_B |h|^2 e^{-pu}.
\ee
Thus, if $h$ is square integrable aginst the weight $e^{-pu}$, the norms $\|h\|^2_s$ decrease like $e^{-ps}$.
 
These norms form  a decreasing family of norms which for $s=0$ is the standard unweighted $L^2$-norm, and which for any $s$ is equivalent to the standard $L^2$-norm. We also let the norm without subscript be the norm in $H^2(B,e^{-u}dm)$,
$$
\|h\|^2=\int_B |h|^2 e^{-u}.
$$
In spite of its simplicity the following proposition is crucial.
\begin{prop} If $h$ lies in $H^2(B,e^{-u}dm)$, then
$$
\|h\|^2= 2\int_0^\infty e^s \|h\|^2_s ds +\|h\|^2_0.
$$
More generally, if $0<p<2$, then
\be
\int_B |h|^2 e^{-pu} = a_p\int_0^\infty e^{ps} \|h\|_s^2 +b_p \|h\|^2_0.
\ee
\end{prop}
For the proof we use the following lemma.
\begin{lma} If $x<0$
and  $0<p<2$,
$$
\int_0^\infty e^{ps} e^{-2\max(x+s,0)}ds +1/p= C_p e^{-px}.
$$
\end{lma}
\begin{proof}
$$
\int_0^\infty e^{ps} e^{-2\max(x+s,0)}ds =\int_0^{-x} e^{ps} ds +
 e^{-2x}\int_{-x}^{\infty} e^{(p-2)s}ds = (1/p+1/(2-p))e^{-px}-1/p.
$$
\end{proof}
To prove the proposition we apply the lemma with $x=u$
$$
C_p\int_B |h|^2 e^{-pu}=\int_0^\infty e^{ps}\int_B |h|^2 e^{-2u_s} ds +1/p\int_B|h|^2 ,
$$
which immediately gives (2.2). 
 \qed

\bigskip

\bigskip

We finally quote a particular case of a result from \cite{2Berndtsson} that is the most important ingredient in the proof of Theorem 1.1. We let $D$ be a domain in $\C$ and let $\E:=H^2(B,dm)\times D$.  This can be regarded as a trivial vector bundle, of infinite rank, over $D$. Let for $\zeta$ in $D$, $v_\zeta$ be bounded plurisubharmonic functions on $B$.
\begin{thm} With notation as above, define an hermitean metric on $\E$ by
\be
\|h\|^2_\zeta=\int_B |h|^2 e^{-v_\zeta}
\ee
Assume  that $v_\zeta$ is plurisubharmonic on $D\times B$. Then the curvature of the metric (2.3) on $\E$ is nonnegative.
\end{thm}

This result applies in particular to our present setting, with $D$ equal to the right half plane. Then
$$
v_\zeta= 2 u_s
$$
for $s=\Re\zeta$. By the definition of $u_s$, $u_{\Re\zeta}$ is  plurisubharmonic on  $D\times B$. Hence, by Theorem 2.3,  the norms $\|h\|_s$ on $H^2(B,dm)=H_0$ define a metric on $\E$ of positive curvature. Moreover, they depend only on $s=\Re\zeta$. 

\subsection{Infinite rank vector bundles}
The theory of infinite rank vector bundles in general, and their curvature in particular, is a bit subtle and perhaps not yet completely developed,see \cite{Lempert-Szoke} and \cite{2Berndtsson} . In particular, it is not clear if local triviality is really a reasonable requirement on vector bundles of infinite rank. In the present situation however our vector bundle is even globally trivial, so the notion of vector bundle itself does not pose any problem. In this paper we adopt the definition  that the bundle is positively curved if its dual bundle is negatively curved.  We define negative curvature as meaning that the logarithm of the norm of any local holomorphic section is plurisubharmonic. It is well known that for bundles of finite rank, this is equivalent to negative curvature.

Suppose that a  metric on a vector bundle (possibly of infinite rank) can be diagonalized, with diagonal entries $\omega_j(\zeta)$. Then the bundle is a direct sum of line bundles, so it follows that it is negatively curved if and only if $\log\omega_j$ are subharmonic for all $j$. Similarily, it is positively curved when  $\log\omega_j$ are superharmonic, and flat when $\log\omega_j$ are harmonic. If in addition, as in our present setting, the norms depend only on $s=\Re\zeta$, positive curvature means that
all $\omega_j(s)$ are logconcave. This is  our substitute for the logconcavity of $|\Omega(s)|$, mentioned in the introduction. 

\section{Integrals of quadratic forms and the proof of Theorem 1.1}
We continue the discussion from the previous section and specialize to $D=U$, the right half plane. Of course the integrability of $e^{-u}$ and $e^{-pu}$ is strongly connected to the finiteness of norms of the type in Proposition 2.1, when $h$ is a general element in $H^2(B,dm)$, e g $h=1$. We are therefore led to the question when integrals like
$$
\int_0^\infty e^s \|h\|^2_s ds
$$
converge, under the assumption that the family of norms $\|\cdot\|_s$ is positively curved. We will first look at this problem abstractly, when $\|\cdot\|_s$ is a family of equivalent norms on an abstract Hilbert space $H_0$. There is a complication  arising from the fact that our vector bundle has infinite rank. For motivation we start by discussing the case of bundles of finite rank, and first of all even the case of rank one. Then  we have a one dimensional vector space, $E$, and the norms can be written
\be
\|h\|^2_s=\|h\|^2_0 e^{-k(s)}.
\ee
By the discussion in subsection 2.1, positive curvature means that $k(s)$ is convex.
\begin{prop} Let $k(s)$ be convex for $0\leq s$. Then 
$$
\int_0^\infty e^s e^{-k(s)}ds<\infty
$$
(if and) only if $\lim_{s\rightarrow \infty} k(s)/s>1$
\end{prop}
\begin{proof} We may of course assume that, as in (3.1), $k(0)=0$. Then $k(s)/s$ is increasing so its limit as $s$ goes to infinity exists. If the limit is not greater that one, then, since the quotient increases,  $k(s)\leq s$ for all $s>0$. Hence the integral diverges. The other direction is obvious.
\end{proof}
The next theorem is an analog of this statement for bundles of finite rank.
\begin{thm}
Let $\|\cdot\|_s$ be a family of Hilbert norms on some finite dimensional vector space $E$ such that the induced hermitean metric on the trivial vector bundle $\E:=U\times E$, $\|\cdot\|_{\Re \zeta}$, has positive curvature over $U$. Then the integrals
$$
\int_0^{\infty} e^s \|h\|^2_s ds
$$
converge for all $h$ in $E$ if and only if there are $\epsilon>0$ and $s_0$ such that
\be
 \|h\|^2_s\leq e^{-(1+\epsilon)s} \|h\|^2_0
\ee
for $s>s_0$ and any $h$ in $E$.\end{thm}
\begin{proof}
One direction is of course clear; if (3.2) holds the integral converges. So, assume that (3.2) does not hold. Then, for any $\epsilon>0$ we can find  $s>1/\epsilon$ and some $h$ in $E$ such that
\be
 \|h\|^2_s > e^{-(1+\epsilon)s} \|h\|^2_0.
\ee
By the spectral theorem we can choose an orthonormal basis, $e_j$, for $\|\cdot\|_0$ that diagonalises  $\|\cdot\|_s$. If 
$$
h=\sum c_j e_j
$$
we can write
$$
\|h\|^2_0=\sum |c_j|^2 \quad \text{and}\quad \|h\|^2_s=\sum |c_j|^2e^{-s\lambda_j}
$$
since the eigenvalues are positive. By (3.3), at least one $\lambda_j$ - say $\lambda_0$ - is smaller than $(1+\epsilon)$. 

\bigskip

We now define another family of norms $|\cdot|_t$ for $0\leq t\leq s$ by
$$
|h|^2_t=\sum |c_j|^2 e^{-t\lambda_j}
$$
Since $e_j(\zeta)=e_j e^{\lambda_j\zeta/2}$ defines a global holomorphic orthonormal frame,  $|h|^2_{\Re \zeta}$ defines a hermitean metric on $\E$ of zero curvature. Moreover the new norms agree with the previous ones for $t=0$ and $t=s$. By the maximum principle for positive metrics (see e g \cite{Berman-Keller}, Lemma 8.11) we have
$$
\|h\|^2_t\geq |h|^2_t
$$
for $0\leq t\leq s$. Choose $h=e_0$. Then we conclude that
$$
\|e_0\|^2_t\geq |e_0|^2_t\geq \|e_0\|^2_0 e^{-(1+\epsilon)t}.
$$
Therefore
$$
\int_0^s e^t\|e_0\|^2_t dt\geq \|e_0\|^2_0\int_0^s e^{-\epsilon t} dt\geq
 \|e_0\|^2_0 (1-e^{-1})/\epsilon,
$$
since $s>1/\epsilon$. Since $\epsilon$ can be taken arbitrarily small we see that there is no constant such that
\be
\int_0^{\infty} e^s \|h\|^2_s ds\leq C\|h\|_0^2
\ee
for all $h$ in $E$. 
On the other hand, if all integrals
$$
\int_0^\infty e^s \|h\|^2_s ds
$$
did converge, these integrals would define a new norm on $E$. Since all norms on a finite dimensional vector space are comparable, there would be a constant $C$ so that (3.4)    holds. This completes the proof. 

\end{proof}

\bigskip

A statement as clean as in the theorem can not hold for bundles of infinite rank. To see this, consider again the rank one case. If the integral in Proposition 3.1 converges, we have $k(s)/s>(1+\epsilon)$ if $s$ is sufficiently large, but there is of course no uniformity in $\epsilon$ or in how large $s$ has to be chosen. Now take e g  $E$ to be $l^2$, with standard basis $e_j$. Let the norms be defined by 
$$
\|e_j\|_s^2= e^{-k_j(s)}
$$
with $k_j$ convex. If the $k_j$ are chosen with, say, $k_j(s)=0$ for $s<j$ we see that the claim of Theorem 3.2 fails. It seems possible that in the infinite rank case one always has that the space of $h$ in $E$ such that
$$
\limsup \log\|h\|^2_s/s<-(1+\epsilon),
$$
for some $\epsilon>0$, is dense for the norm $\|\cdot\|_0$. This would suffice for our application, but I have not been able to prove that, (see however the last section),  so we will get by with a weaker statement.

\begin{thm} Let $H_0$ be a (separable) Hilbert space equipped with a decreasing family of equivalent Hilbert norms $\|\cdot\|_s$ of positive curvature, defining new Hilbert spaces, $H_s$. Let $H$ be the subspace of $H_0$ of elements $h$ such that
$$
\|h\|^2:=\int_0^\infty e^s \|h\|^2_s ds<\infty.
$$
Then, for any $h$ in $H$,  $\epsilon>0$  and $s>1/\epsilon$ there is an element $h_s$ in $H_0$ such that
\be
\|h-h_s\|^2_0\leq 2\epsilon \|h\|^2,
\ee
 and
\be
\|h_s\|^2_s\leq e^{-(1+\epsilon)s}\|h\|^2_0.
\ee
\end{thm}
\begin{proof} Take $\epsilon>0$ and $s>1/\epsilon$. By assumption there is a bounded linear operator $T_s$ on $H_0$ such that
$$
\langle u,v\rangle_s=\langle T_su,v\rangle_0.
$$
By the spectral theorem (see \cite{Reed-Simon}) we can realize our Hilbert space $H_0$ as an $L^2$-space over a measure space $X$,  with respect to some positive measure $d\mu$, in such a way that 
$$
\|h\|^2_0=\int_X |h|^2 d\mu(x)
$$
and 
$$
\|h\|^2_s=\int_X |h|^2 e^{-s\lambda(x)} d\mu(x).
$$
We define $h_s$ by $h_s=\chi(x)h$, where $\chi$ is the characteristic function of the set $\lambda>(1+\epsilon)$. Let $w=h-h_s$. Clearly
$$
\|h_s\|^2_s=\int_{\lambda>1+\epsilon}|h|^2 e^{-s\lambda(x)}d\mu(x)\leq
e^{-(1+\epsilon)s}\int_X |h|^2d\mu= e^{-(1+\epsilon)s}\|h\|^2_0.
$$
Hence (3.5) is satisfied. For (3.4) we will again use comparison with a flat family of metrics, which now are defined for $0\leq t\leq s$ by
\be
|h|_t^2=\int_X |h|^2 e^{-t\lambda(x)}d\mu(x).
\ee
This is again a flat metric in the sense that any element $h$ in $H_0$ can be extended holomorpically as $h_\zeta=h e^{\zeta\lambda/2}$ in such a way that
$$
\|h_\zeta\|^2_{\Re\zeta}
$$
is constant. Since $|h|_t$ coincides with $\|h\|^2_t$ for $t=0$ and $t=s$ we can apply the maximum principle for positively curved metrics again. (This is perhaps not completely obvious since we have bundles of infinite rank, but we postpone the motivation for this to the end of the proof.) Hence
$$
\|h\|_t^2\geq|h|^2_t
$$
for $t$ between 0 and $s$. Since moreover $w$ and $h_s$ are orthogonal for the scalar product defined by $|\cdot|_t$,
$$
\int_0^s e^t\|h\|^2_t dt\geq\int_0^s e^t|h|^2_t dt\geq \int_0^s e^t|w|^2_t dt.
$$
By the definition of $w$ 
$$
|w|^2_t\geq e^{-t(1+\epsilon)}\|w\|^2_0.
$$
Hence
$$
\int_0^s e^t|w|^2_t dt\geq \int_0^s e^{-\epsilon t}dt\|w\|^2_0\geq 1/(2\epsilon)\|w\|^2_0,
$$
since $s>1/\epsilon$. All in all
$$
\|w\|^2_0\leq 2\epsilon \|h\|^2
$$
so we have proved (3.4). It only remains to motivate why the maximum principle holds for our bundle of infinite rank. It suffices to prove the opposite inequality for negatively curved bundles. Recall that
$$
|h|_t^2=\int_X |h|^2 e^{-t\lambda(x)}d\mu(x).
$$
Since the the $s$-norm and the 0-norm are equivalent, $\lambda$ is bounded. We can therefore approximate $\lambda$ arbitrarily well from above by a sequence of finite valued step functions $\lambda_j $. This decomposes $E$ as a direct sum of finitely many subspaces $E_l$ on which
$$
\int_X |h|^2 e^{-t\lambda_j(x)}d\mu(x)=|h|^2_0 e^{-t\mu_l}.
$$
By choosing orthonormal bases in each summand we can this way exhaust our bundle by finite rank subbundles on each of which the flat metric is still flat. It then suffices to apply the finite rank maximum principle to these subbundles (on which the other metric is also negative), and pass to the limit.

\end{proof}
\subsection{End of the proof of Theorem 1.1}
We shall now see how Theorem 1.1 follows from Theorem 3.3. We take $H_0=H^2(B,dm)$ and $H_s$ defined as in section 2. Assume that 
$$
\int_B e^{-u}<\infty.
$$
Then $h=1$ lies in $H$ (defined in Theorem 3.3) by Proposition 2.2. Take $\epsilon>0$ so small that
$$
2\epsilon\|h\|^2=2\epsilon\int_B e^{-u}=:\epsilon_0,
$$
where $\epsilon_0$ is chosen so that
if $w$ is holomorphic and 
$$
\int_B |w|^2\leq \epsilon_0,
$$
then
$$
\sup_{B/2}|w|\leq 1/10.
$$
Then $h_s=1-w$ satisfies $1<2|h_s|^2$ on $B/2$ so
$$
\int_{B/2} e^{-2u_s}\leq 2\|h_s\|^2_s\leq e^{-s(1+\epsilon)}\|h\|^2_0.
$$
Take $p=1+\epsilon/2$, multiply by $e^{ps}$ and integrate from 0 to infinity. By Lemma 2.3 we get that
$$
\int_{B/2} e^{-pu} <\infty.
$$
We also note that the only requirement on $\epsilon$ was that it satisfy
$$
\epsilon\leq\epsilon_0/(2\int_B e^{-u}),
$$
so we have also proved the last part of Theorem 1.1.
\subsection{A last comment on Theorem 3.2}
We have seen that Theorem 3.2 cannot hold in the case of Hilbert spaces of infinite rank, but as mentioned above a variant, weaker that Theorem 3.2 but stronger than Theorem 3.3, might hold. We shall now see that in our particular situation this is actually the case, and that this fact is in some sense 'equivalent' to the openness conjecture. Let $H'$ be equal to the space pf $h$ in $H_0$ such that there is an $\epsilon>0$ with
$$
\limsup\log \|h\|_s^2/s<-(1+\epsilon).
$$
It is easy to see that this is a subspace of $H_0$, and the question is if this subspace is dense in $H_0$. We shall now see that when $H_0=H^2(B,dm)$ and 
$$
\|h\|^2_s=\int_B |h|^2 e^{-2u_s}
$$
it follows from the openness conjecture that this is true. 
\begin{prop} Let $u\leq 0$ be plurisubharmonic in the unit ball, such that
$$
\int_B e^{-u}<\infty.
$$
Then any square integrable  holomorphic function in the ball can be arbitrarily well approximated in the $L^2$-norm by holomorphic functions $h_j$ satisfying
$$
\int_B |h_j|^2 e^{-p_ju} < \infty
$$
for some $p_j>1$ (depending on $j$).
\end{prop}
By (2.1)
$$
\int_B |h_j|^2 e^{-2u_s}\leq e^{-ps} \int_B|h|^2 e^{-pu},
$$
so $h_j$ lies in the space $H'$. Therefore the corollary implies that $H'$ is dense in $H_0$ in our particular situation. The proof of Proposition 3.4, uses Theorem 1.1, but we omit the details.

\def\listing#1#2#3{{\sc #1}:\ {\it #2}, \ #3.}

\end{document}